\documentclass[12pt]{amsart}
\usepackage[english,activeacute]{babel}
\usepackage{amsmath,amsfonts,amssymb,amstext,amsthm,amscd,mathrsfs,amsbsy}
\usepackage[titletoc]{appendix}
\usepackage{lineno} 
\usepackage{epsfig}  
\usepackage{color}  
\usepackage{setspace}
\newtheorem{teo}{Theorem}[section]
\newtheorem{pro}[teo]{Proposition}
\newtheorem{coro}[teo]{Corollary}
\newtheorem{lem}[teo]{Lemma}

\theoremstyle{definition}
\newtheorem{defi}[teo]{Definition}
\newtheorem{exam}[teo]{Example}
\newtheorem{rem}[teo]{Remark}

\newcommand{\K}{\mathbb K}

\newcommand{\N}{\mathbb N}

\newcommand{\Q}{\mathbb Q}

\newcommand{\F}{\mathbb F}

\newcommand{\m}{\mathfrak m}
\newcommand{\Po}{\K[x_1,\ldots,x_N]}
\newcommand{\de}{\mbox{depth}}

\newcommand{\D}{\Delta}
\newcommand{\Dn}{\Delta(n)}

\newcommand{\p}{\mathfrak{p}}

\begin{document}

\title{Cohen-Macaulayness of triangular graphs}

\author{Hernan de Alba, Walter Carballosa, \\
Daniel Duarte, Luis Manuel Rivera}

\maketitle
\begin{abstract}

We study the Cohen-Macaulay property of triangular graphs $T_n$. We show that $T_2$, $T_3$ and $T_5$ are Cohen-Macaulay graphs, and that $T_4$, $T_6$, $T_8$ and $T_n$ are not Cohen-Macaulay graphs, for $n \geq 10$. Finally, we prove that over fields of characteristic zero $T_7$ and $T_9$ are Cohen-Macaulay.

\end{abstract}


\section{Introduction}
Let $R=\Po$ be the polynomial ring over $\K$, where $\K$  is any field. Let $G$ be a simple graph with vertex set $V(G)=\{v_1,\ldots,v_N\}$ and edge set $E(G)$. We identify the vertex $v_i$ with the variable $x_i$. The {\it edge ideal} $I(G)$ of $G$ is the ideal $\langle x_ix_j \colon~\{v_i,v_j\}\in E(G)\rangle$. The graph $G$ is called 
\textit{Cohen-Macaulay} over $\K$ if $R/I(G)$ is a Cohen-Macaulay ring. According to \cite{HHZ}, it is unlikely to have a general classification of Cohen-Macaulay graphs. This situation has led to an extensive study of the Cohen-Macaulay property of particular families of graphs (see, for instance, \cite{EV, Fe, HYZ1, HYZ2, HHZ, VTG, Vi2}).

In this note we study  the Cohen-Macaulayness of triangular graphs. The {\it triangular graph} $T_n$ is the simple graph whose vertices are the $2$-subsets of an $n$-set, $n\geq 2$, and two vertices are adjacent if and only if their intersection is nonempty. 
 It is known that $T_n$ is isomorphic to the Johnson graph $J(n, 2)$, which is in turn the $2$-token graph of the complete graph $K_n$ (see, for instance, \cite{CFLR, FFHH, Jo}). In addition, the complement of $T_n$ is isomorphic to the Kneser graph $K(n, 2)$ and the complement of $T_5$ is isomorphic to the Petersen graph. 

Our main theorem (Theorem~\ref{main}) states that $T_2$, $T_3$ and $T_5$ are Cohen-Macaulay, and that $T_4$, $T_6$, $T_8$ and $T_n$ are not Cohen-Macaulay graphs, for $n \geq 10$.  In addition, it is proved that over fields of characteristic zero $T_7$ and $T_9$ are Cohen-Macaulay.

This note is organized as follows. We start by recalling the basic definitions and results regarding Cohen-Macaulay
graphs that we need. Next, in Section \ref{characterization}, we first prove that $T_n$ 
is unmixed for every $n\in\N$. Later, we give a characterization for the Cohen-Macaulay property of $T_n$ that follows from 
Reisner criterion (Proposition \ref{reduction}). In Section \ref{3 5 7 9}, we first prove that $T_3$ and $T_5$ are Cohen-Macaulay. Next, using a computer algebra system, we compute explicit regular sequences to show that $T_7$ and $T_9$ are Cohen-Macaulay over fields of characteristic zero. Finally, in Section \ref{n even n odd}, we show that $T_4$, $T_6$, $T_8$ and $T_n$ are not Cohen-Macaulay graphs, for $n \geq 10$.

When investigating about the Cohen-Macaulayness of $T_n$, we computed several regular sequences using symmetric
polynomials and we noticed that there was a certain pattern on how these sequences behave as $n$ increases. We
later realized that those patterns also appeared for the edge subring of any simple graph. We conclude this note
with an appendix in which we present an explicit regular sequence of a particularly nice shape for Cohen-Macaulay graphs.
To that end, we first prove the existence of an explicit homogeneous system of parameters using elementary
symmetric polynomials.


\section{Cohen-Macaulay graphs and Cohen-Macaulay simplicial complexes}\label{cm graphs}

Let $R=\Po$ be the polynomial ring over the field $\K$.  Let $\m=\langle x_1,\ldots,x_N \rangle$ and let $I$ be a graded ideal of $R$.
The \textit{depth} of $R/I$ is defined as the largest integer $r$ such that there is a homogeneous sequence $\{h_1, \dots, h_r\} \subset \m$,
such that $h_1$ is not a zero divisor of $R/I$ and $h_i$ is not a zero divisor of $R/\langle I,h_1, \dots, h_{i-1}\rangle$, for every $i\geq2$.

\begin{defi}
We say that $R/I$ is a \textit{Cohen-Macaulay ring} (CM ring for short) if $\de(R/I)=\dim(R/I)$, 
where $\dim(R/I)$ denotes the Krull dimension of $R/I$.
\end{defi}
Let $G$ be a simple graph with vertex set $V(G)=\{v_1,\dots, v_N\}$ and edge set $E(G)$. We identify each vertex $v_i$ with the
variable $x_i$ in $R$. The \textit{edge ideal} $I(G)$ of $G$ is the ideal $\langle x_ix_j\colon~\{v_i,v_j\}\in E(G)\rangle$. The ring $R/I(G)$
is called the \textit{edge subring} of $G$. We say that $G$ is a \textit{Cohen-Macaulay graph over} $\K$ if $R/I(G)$ is CM. We say that $G$ is  a \textit{Cohen-Macaulay graph} if $G$ is CM over any field.

A set $U$ of vertices in a graph $G$ is an \emph{independent set} of vertices if no two vertices in $U$ are adjacent;  a \emph{maximal independent set} is an independent set which is not a proper subset of any independent set in $G$. The \emph{independence number} of $G$ is the number of vertices in a largest independent set in $G$. It is well known that the Krull dimension of $R/I(G)$ is equal to the independence number of $G$ (see \cite{isivi, Vi1}).

Let $\Delta$ be a simplicial complex on the vertex set $V=\{v_1,\ldots,v_N\}$, i.e., $\Delta$ is a family of subsets of $V$ closed under taking subsets and such that $\{v_i\}\in \Delta$, for every $i$. The elements of $\Delta$ are called faces of $\Delta$. The dimension of a face $F \in \Delta$ is $|F|-1$. The dimension of $\D$ is the largest dimension of its faces. As before, we identify $v_i$ with $x_i$. The 
\textit{Stanley-Reisner ideal} $I_{\Delta}$ of $\Delta$ is the ideal generated by all monomials $x_{i_1}\cdots x_{i_r}$ such that $\{v_{i_1},\dots,v_{i_r}\}\notin\Delta$. We say that $\Delta$ is a \textit{Cohen-Macaulay simplicial complex} over $\K$ if $R/I_{\Delta}$ is a CM ring. We say that $\Delta$ is a \textit{Cohen-Macaulay simplicial complex} if $\Delta$ is CM over any field.

\begin{rem} 
Let $G$ be a simple graph.
	\begin{enumerate}
	\item Let $\Delta_G$ be the simplicial complex formed by the independent sets of $G$ (this is a simplicial complex since every subset of an independent set is also independent). Hence $I(G)=I_{\Delta_G}$. Therefore, $G$ is a CM graph if and only if $\Delta_G$ is a CM simplicial complex.
	\item A {\it clique} of a graph $G$ is a subset $S\subseteq V(G)$ such that the graph induced by $S$ is a complete graph. Let $\Delta(G)$ be the simplicial complex formed by all cliques of $G$ and let $\overline{G}$ be the complement graph of $G$. Notice that $\Delta(\overline{G})=\Delta_G$: every clique of $\overline{G}$ is an independent set of $G$ and vice versa.
	\end{enumerate}
\end{rem}

\begin{defi}
Let $\Delta$ be a simplicial complex and $F\in\Delta$. The {\it link}  of $F$ in $\D$, denoted $lk_{\D}(F)$, is the simplicial complex  
$\{H\in\Delta \colon~H\cap F=\emptyset\mbox{ and } H\cup F\in\Delta\}.$
We will often denote the link of $F$ in $\Delta$ just as $lk(F)$ if there is no risk of confusion.
\end{defi}
The CM property of a graph can be determined by the following homological criterion (see \cite{reis}). 
\begin{teo}(\textit{Reisner's criterion})\label{Reisner criterion}
Let $\Delta$ be a simplicial complex. The following conditions are equivalent:
\begin{itemize}
\item[(a)]$\Delta$ is Cohen-Macaulay over $\K$.
\item[(b)]$\widetilde{H_i}(lk(F);\K)=0$, with $F\in\Delta$ and $i<\dim lk(F)$.
\end{itemize}
\end{teo}

\begin{coro}\label{Reisner criterion dim 1}
If $\Delta$ is a 1-dimensional simplicial complex, then $\Delta$ is CM if and only if $\Delta$ is connected.
\end{coro}

We will also need a result relating the CM property of a simplicial complex to some property of the $h-$vector
of the simplicial complex.

\begin{defi}
Let $\Delta$ be a simplicial complex of dimension $d$.
\begin{itemize}
\item[{i.}] The $f-$vector of $\D$ is defined as $f(\Delta)=(f_{-1},f_0,\dots,f_d)$,
where $f_{-1}=1$ and $f_i$ denotes the number of faces of dimension $i$ of $\D$, for $i\geq0$.
\item[{ii.}] The $h-$vector of $\D$ is defined as $h(\Delta)=(h_0, \dots, h_{d+1})$, where
\[
h_k=\sum_{i=0}^{k}(-1)^{k-i}\binom{d+1-i}{k-i}f_{i-1},
\]
and $0\leq k \leq d+1$.
\end{itemize}
\end{defi}

\begin{teo}\cite[Chapter II, Corollary 3.2]{stanley}\label{h vector}
Let $\D$ be a simplicial complex of dimension $d$. If $\D$ is CM, then
$h_i(\D)\geq0$, for $0\leq i \leq d+1$.
\end{teo}


\section{A characterizacion for the CM property of $T_n$}\label{characterization}

The {\it triangular graph} $T_n$ is the simple graph having as vertices the $2$-subsets of a $n$-set, $n\geq 2$, and two vertices are adjacent 
if and only if their intersection is nonempty. The triangular graph $T_4$ is shown in Figure~\ref{johnson42}. 
\begin{figure}[h]
\hfill
\begin{minipage}[t]{.45\textwidth}
\begin{center}
\epsfig{file=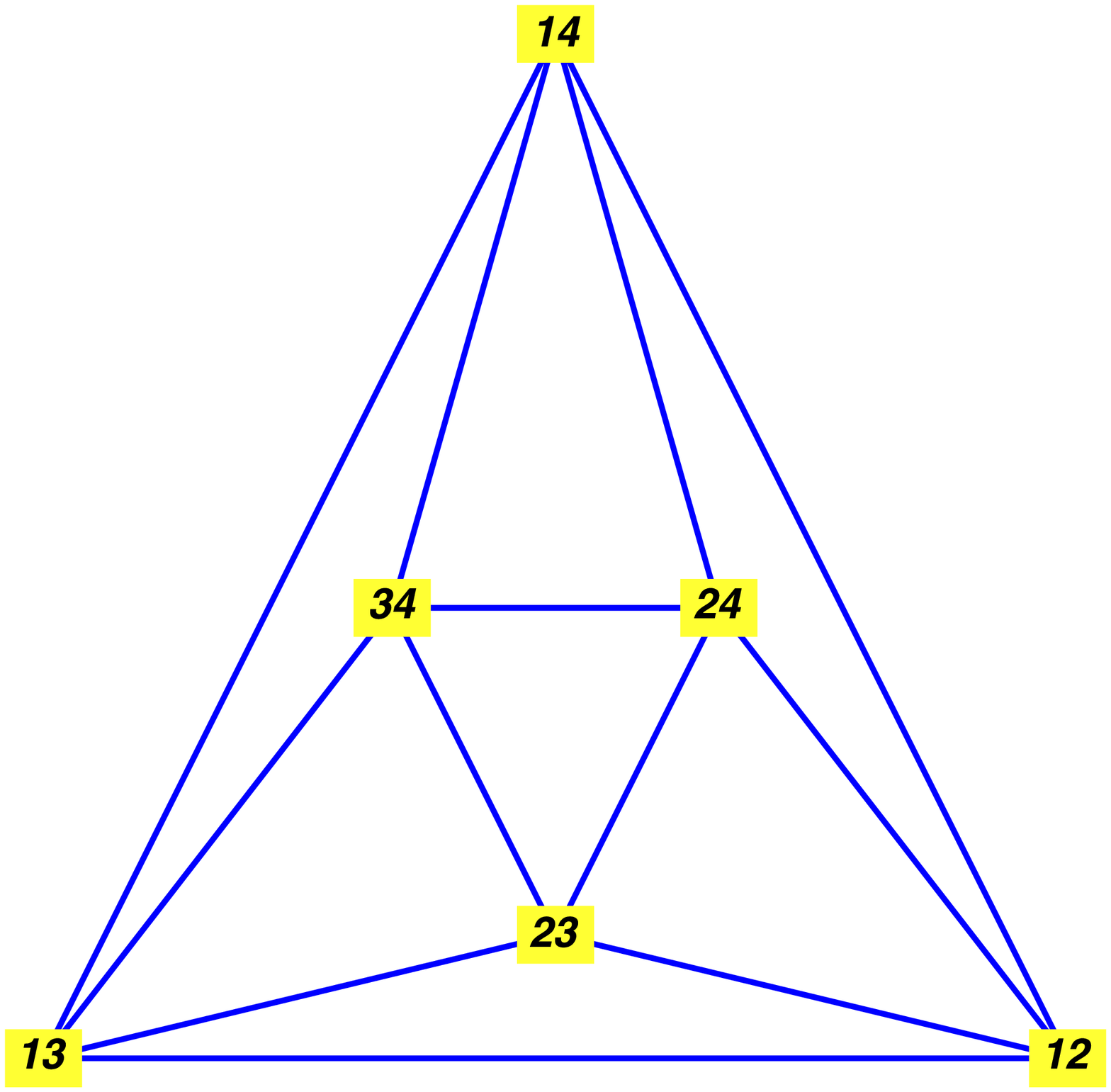, scale=0.3} 
\caption{$T_4$.}
\label{johnson42}
\end{center}
\end{minipage}
\hfill
\begin{minipage}[t]{.45\textwidth}
\begin{center}
\epsfig{file=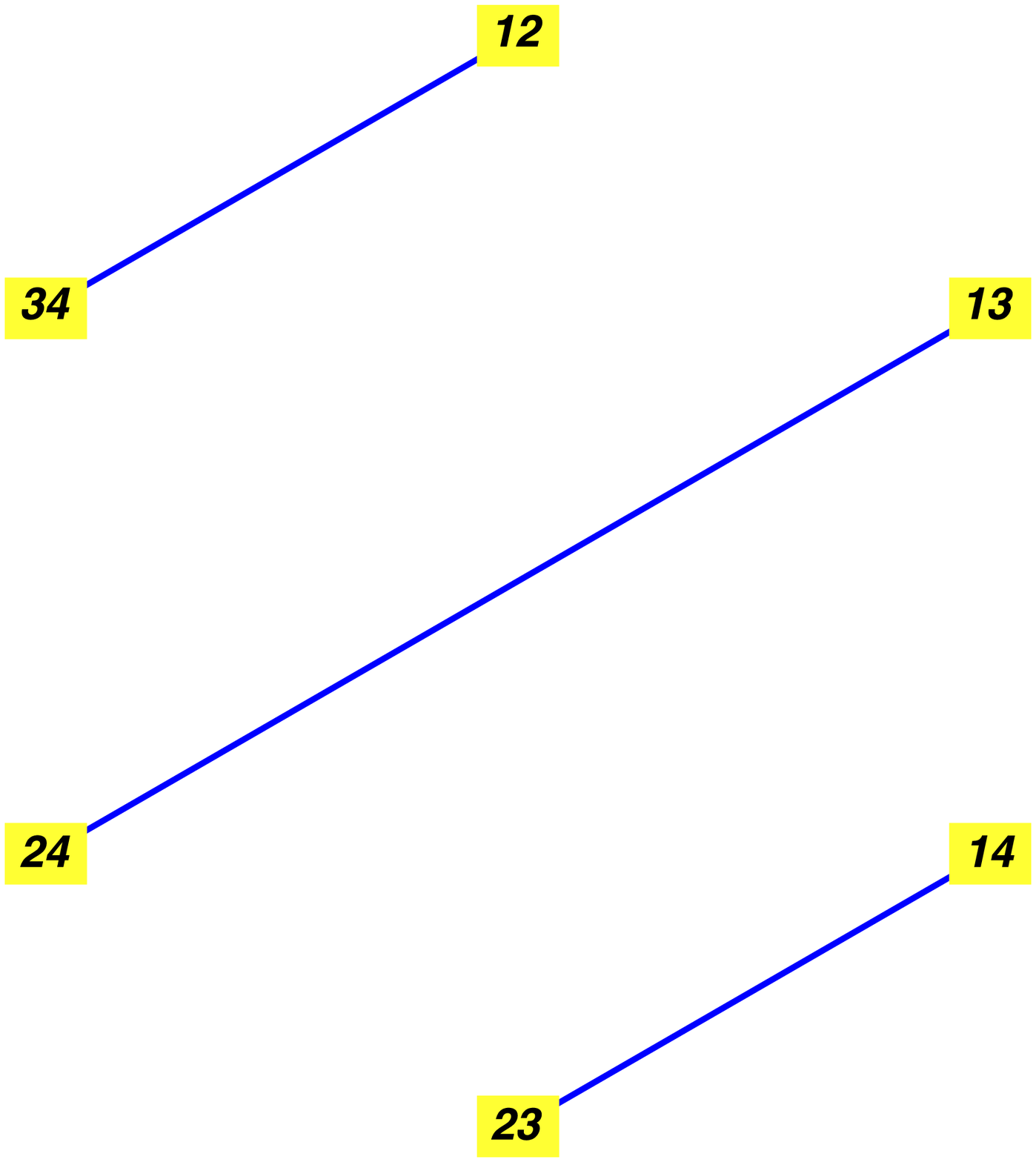, scale=0.3} 
\caption{$\overline{T_4}$.}
\label{c-johnson42}
\end{center}
\end{minipage}
\hfill
\end{figure}

We denote by $(ij)$ the vertices of $T_n$, where $1\leq i<j\leq n$, and by 
$\Dn$ the simplicial complex of independent sets of $T_n$. If $n<2$ we define $\Dn=\emptyset$.

A graph $G$ is {\it unmixed} if any two maximal independent sets of $G$ have the same cardinality. 
Since every CM graph is unmixed, the following proposition is relevant.

\begin{pro}\label{Johnson unmixed}
Every triangular graph $T_n$ is unmixed.
\end{pro}
\begin{proof}
It is well known that the independence number of $T_n$ is $\lfloor{n/2}\rfloor$. We prove, by contradiction,  that any maximal independent set in $T_n$ 
has $\lfloor{n/2}\rfloor$ vertices. Let $A$ be any maximal independent set of $T_n$ and suppose that $|A| < \lfloor{n/2}\rfloor$. Notice that there are $n-2|A|$ elements in $\{1,\dots,n\} \setminus \cup A$, with $n-2|A| >1$. Therefore, we can take a 2-set, say $z$, from $\{1,\dots,n\} \setminus \cup A$ to construct the independent set $A'=A\cup\{z\}$, which is a contradiction. 
\end{proof}

We need the following lemma to give a characterization for the Cohen-Macaulay property of 
the triangular graph $T_n$. 

\begin{lem}\label{isom face}
Let $F\in\Dn$ be any face such that $|F|=m$, where $n \geq 2$ and $m\geq0$. Then we have the following identification of simplicial complexes:
\[
lk_{\Dn}(F)\cong \D(n-2m).
\]
\end{lem}
\begin{proof}
If $F=\emptyset$ the statement holds by definition of $lk_{\Dn}(F)$.  If $m=  \lfloor{n/2}\rfloor$, then $n=2m$ or $n=2m+1$ which implies that $lk_{\Dn}(F)=\emptyset=\D(n-2m)$ in both cases. Suppose $1 \leq m < \lfloor{n/2}\rfloor$. Assume 
$F=\{(i_1j_1),\dots,(i_mj_m)\}$. Let $A=\{1,\dots,n\}\setminus\{i_1,\dots,i_m,j_1,\dots,j_m\}$.
Since $F$ is an independent set of $T_n$ we have that $|A|=n-2m \geq 2$. 
Notice that $lk_{\Dn}(F)$ consists of every independent set of $T_n$ formed by elements $(ij)$ such that 
$i, j\in A$, and $i\neq j$. Now observe that the set of independent sets formed with couples $(ij)$ with $i,j\in A$, $i\neq j$ can be 
identified with $\D(n-2m)$.
\end{proof}

\begin{pro}\label{reduction}
Let $n \geq 2$. Assume $n$ is odd (resp. even). The simplicial complex $\Dn$ is CM if and only if $\widetilde{H_i}(\D(l); \K)=0$ for every $l\leq n$, with $l$ odd (resp. even), and for every $i<\dim(\D(l))$.
\end{pro}
\begin{proof}
Assume that $n$ is odd. Suppose that $\Dn$ is CM.  Choose any odd number $l$ such that $3\leq l\leq n$. By Lemma \ref{isom face}, 
$\D(l)\cong lk_{\Dn}(F)$ for any face $F\in\Dn$ such that $|F|=(n-l)/2$. Thus, $\widetilde{H_i}(\D(l); \K)=\widetilde{H_i}(lk_{\Dn}(F); \K)=0$, 
for every $i<\dim lk_{\Dn}(F) =\dim(\D(l))$, according to Reisner's criterion (Theorem \ref{Reisner criterion}).

To prove the other implication, let $F\in\Dn$ be such that $|F|=m$, with $m\geq0$. By Lemma \ref{isom face}, $lk_{\Dn}(F)\cong \D(n-2m)$.
Since $n$ is odd, $n-2m$ is also odd and $n-2m\leq n$. By the hypothesis, $\widetilde{H_i}(lk_{\Dn}(F); \K)=\widetilde{H_i}(\D(n-2m); \K)=0$, for every $i<\dim(\D(n-2m))=\dim lk_{\Dn}(F)$. By Reisner's criterion, $ \Dn$ is CM. The proof is completely analogous for $n$ even.
\end{proof}

\begin{coro}\label{lema walter}
Suppose that there exists an odd (resp. even) integer $n_0$ such that $T_{n_0}$ is not CM. Then $T_n$ is not CM for every odd (resp. even) $n\geq n_0$.
\end{coro}


\section{The Cohen-Macaulay property of  $T_3, T_5, T_7$ and $T_9$}\label{3 5 7 9}

\begin{pro}\label{3,5 are CM}
$T_3$ and $T_5$ are CM graphs.
\end{pro}
\begin{proof}
Since $T_3$ is a complete graph, by Example \ref{complete graph is CM} below, $T_3$ is CM.
Now consider the following path in $\D(5)$:
\[
(12),(34),(25),(14),(35),(24),(13),(45),(23),(15).
\]
This path passes through all vertices in $\D(5)$, hence it is connected. Since the independence number of
$T_5$ is $2$, the simplicial complex $\D(5)$ is 1-dimensional. By Corollary \ref{Reisner criterion dim 1}
we conclude that $\D(5)$ is CM, that is, $T_5$ is CM.
\end{proof}

To verify that $T_7$ and $T_9$ are CM we used the computer algebra system $\mathtt{SINGULAR}$ 
$\mathtt{4}$-$\mathtt{0}$-$\mathtt{2}$ \cite{DGPS}. One minor difficulty here is to effectively compute the edge
ideal of $T_n$. To that end we use the following remark.

\begin{rem}
The graph $T_n$ can be obtained from $T_{n-1}$ and the complete graph $K_{n-1}$ on the vertices 
$(1\mbox{ }n),(2\mbox{ }n),\dots,(n-1\mbox{ }n)$ by joining the vertex $(i\mbox{ }j)\in V(T_{n-1})$  with the vertices $(i\mbox{ }n)$ and $(j\mbox{ }n)$ of $K_{n-1}$. Then we can compute recursively the edge ideal 
$I(T_n)$: if the edge ideal $I(T_{n-1})$ has been computed, we only need to add the monomials 
corresponding to $(i\mbox{ }j)\sim (i\mbox{ }n)$, $(i\mbox{ }j)\sim(j\mbox{ }n)$, and all the monomials 
coming from $K_{n-1}$.
\end{rem}

Using the previous procedure we computed $I(T_7)\subset R_1=\Q[z_1,z_2,\ldots,z_{21}]$ and 
$I(T_9)\subset R_2=\Q[z_1,z_2,\ldots,z_{36}]$. Using the library $\mathtt{primdec.lib}$ \cite{DLPS}, we compute primary 
decomposition of ideals and we found that the sequence
\begin{align*}
\Big\{\sum_{i=1}^{21}z_i,\sum_{i=1}^{21}z_i^2,\sum_{i=1}^{21}z_i^3 \Big\},
\end{align*}
is a regular sequence of $R_1/I(T_7)$. Similarly, the sequence
\begin{align*}
\Big\{\sum_{i=1}^{36}z_i,\sum_{i=1}^{36}z_i^2,\sum_{i=1}^{36}z_i^3,\sum_{i=1}^{36}z_i^4 \Big\},
\end{align*}
 is a regular sequence of $R_2/I(T_9)$ (see the appendix for a discussion on homogeneous system of parameters
for edge ideals using symmetric polynomials). Since the independence number of $T_7$ and $T_9$ are $3$ and $4$,
respectively, we conclude that
\begin{pro}\label{7,9 are CM}
$T_7$ and $T_9$ are CM graphs over any field of characteristic zero.
\end{pro}


\section{Non-Cohen-Macaulayness of $T_4$, $T_6$, $T_8$ and $T_n$ for $n\geq10$}\label{n even n odd}

In this section we show that $T_n$ is not CM for  $n$ even, $n \geq 4$. We also show
that $T_n$ is not CM for $n$ odd, $n \geq 11$.

\begin{pro}\label{even not CM}
The triangular graph $T_n$ is not CM if $n$ is even, except for $n=2$.
\end{pro}
\begin{proof}
If $n=2$, $T_n$ is a single vertex and so it is CM. Let $n=4$.
The simplicial complex $\D(4)$ is 1-dimensional and non-connected, actually $\D(4)=\overline{T_4}$ (see Figure~\ref{c-johnson42}). By Corollary 
\ref{Reisner criterion dim 1}, $\D(4)$ is not CM. Now, Corollary \ref{lema walter} implies that $T_n$ 
is not CM for every $n\geq4$ with $n$ even.
\end{proof}

Now we turn our attention to $T_n$ for $n\geq11$, $n$ odd.

\begin{lem}\cite[Theorem 6.9.1]{sfa} \label{formula}
The number of faces of dimension $i$ of $\Dn$ is given by the following formula:
\[
f_i=\frac{1}{2^{i+1}}\cdot\frac{n!}{{(i+1)!}(n-2(i+1))!}.
\]
\end{lem}

\begin{pro}\label{non CM n>11}
$T_n$ is not CM for every $n\geq11$,  $n$ odd.
\end{pro}
\begin{proof}
Using the formula of lemma \ref{formula}, we find that 
\begin{align}
f(\D(11))&=(1,55,990,6930,17325,10395)\notag\\
h(\D(11))&=(1,50,780,4280,6220,-936)\notag
\end{align}
Since there is a negative entry in $h(\D(11))$, Theorem \ref{h vector} implies that $T_{11}$ is not CM. Corollary \ref{lema walter} implies that $T_n$ is not CM for every odd $n$, $n\geq11$.
\end{proof}

Putting together the results of the previous sections we obtain the following classification of $T_n$ in terms
of the CM property:

\begin{teo}\label{main}
For triangular graphs $T_n$, the following holds:
\begin{itemize}
\item[(i)] $T_2$, $T_3$ and $T_5$ are CM graphs. 
\item[(ii)] $T_7$ and $T_9$ are CM graphs over any field of characteristic zero.
\item[(iii)] $T_4$, $T_6$, $T_8$ and $T_n$, for $n\geq 10$, are not CM graphs. 
\end{itemize}
\end{teo}
\begin{proof}
The theorem follows from propositions \ref{3,5 are CM}, \ref{7,9 are CM}, \ref{even not CM}, and \ref{non CM n>11}. 
\end{proof}

\begin{rem}
Using $\mathtt{SINGULAR}$ 
$\mathtt{4}$-$\mathtt{0}$-$\mathtt{2}$, we verified that $T_7$ is CM over some fields of positive characteristic, such as
$\F_2$, $\F_3$ and $\F_5$. In every case, we found explicit regular sequences using symmetric polynomials (see appendix). This fact suggests 
that $T_7$ and $T_9$ might be CM over any field, giving a complete classification of $T_n$ in terms of the CM property. 
\end{rem}


\appendix

\section{An explicit regular sequence for CM graphs}

It is well known that for any graded ideal $I\subset\Po$, there exists a homogeneous system of parameters (h.s.o.p. for short) 
for $\Po/I$ (see, for instance, \cite[Proposition 2.2.10]{Vi1}). In this appendix we revisit this result for edge ideals by showing 
the existence of an explicit h.s.o.p. of a particularly nice shape. 

This study was motivated by the following fact. When investigating about the Cohen-Macaulayness of $T_n$, experimental computation 
showed that for small odd values of $n$, the sequence 
\[
\Big\{\sum_{\{x_i\}\in A_1} x_i,\sum_{\{{x_{i_1},x_{i_2}\}}\in A_2} x_{i_1}x_{i_2},\dots,\sum_{\{x_{i_1},\dots,x_{i_d}\}\in A_d} x_{i_1}\cdots x_{i_d} \Big\}
\]
is a regular sequence of the edge subring of $T_n$, where $A_j$ is the set of independent sets of size $j$ in $T_n$  and $d$ is the Krull dimension of the edge subring. Inspired by this fact, we show that for every simple graph  there is a h.s.o.p. for its edge subring having this shape.

Let $G$ be a simple graph on the set of vertices $\{x_1,\dots,x_N\}$ and let $I(G)\subset\Po$ be its edge ideal. 
Let us recall the correspondence between minimal vertex covers of $G$, i.e., complements of maximal independent sets of $V(G)$, and minimal primes of $I(G)$.

\begin{pro}\cite[Proposition 6.1.16]{Vi1} \label{minimal primes}
If $\p$ is an ideal of $\Po$ generated by $C=\{x_{i_1},\ldots,x_{i_r}\}$, then $\p$ is a minimal prime of $I(G)$ 
if and only if $C$ is a minimal vertex cover of $G$.
\end{pro}

\begin{exam}\label{complete graph is CM}
Let $K_N$ be the complete graph on the vertices $\{x_1,\ldots,x_N\}$. Every minimal vertex cover of $K_N$ has the form $K_N\setminus\{x_i\}$, for some $i$. 
By the previous correspondence, every minimal prime of $I(K_N)$ is generated by all variables except $x_i$. It follows that
$\sum_{i=1}^{N}x_i$ is not a zero divisor in $R/I(K_N)$, that is, $1\leq\de(R/I(K_N))\leq\dim(R/I(K_N))=1$. Thus, $K_N$ is CM.
\end{exam}

\begin{exam}
The edge ideal of $T_4$  (Figure~\ref{johnson42}) is given by
\[
I(T_4)=\langle x_{12}x_{13},x_{12}x_{14},x_{12}x_{23},x_{12}x_{24},x_{13}x_{34},x_{14}x_{34},x_{23}x_{34},x_{24}x_{34}\rangle.
\]
Let $R=\K[x_{12},x_{13},x_{14},x_{23},x_{24},x_{34}]$. Let $\sigma_1,\dots,\sigma_6\in R$ be the elementary symmetric polynomials.
Since the independence number of $G$ is $2$, we have that $[\sigma_i]=[0]$ in $R/I(T_4)$, for $i \in \{3, 4, 5, 6\}$. In addition, 
\begin{align*}
[\sigma_1]&=[x_{12}+x_{13}+x_{14}+x_{23}+x_{24}+x_{34}],\\
[\sigma_2]&=[x_{12}x_{34}+x_{13}x_{24}+x_{14}x_{23}].
\end{align*} 
Using lemma \ref{radical symmetric} below, we conclude that
\[
\sqrt{\langle [\sigma_1],[\sigma_2] \rangle}=\sqrt{\langle [\sigma_1], \dots,[\sigma_6] \rangle}=\langle [x_{12}],[x_{13}],[x_{14}],[x_{23}],[x_{24}],[x_{34}] \rangle.
\]
Since $\dim R/I(T_4)=2$, we conclude that $\{[\sigma_1],[\sigma_2]\}$ is a h.s.o.p. for $R/I(T_4)$.
\end{exam}

\begin{lem}\label{radical symmetric}
Let $\sigma_1,\ldots,\sigma_m\in\K[z_1,\ldots,z_m]$ be the elementary symmetric polynomials. Then 
$\sqrt{\langle \sigma_1,\dots,\sigma_m \rangle}=\langle z_1,\dots,z_m\rangle.$
\end{lem}
\begin{proof}
It is enough to consider the following telescopic sum:
\[
z_i^{m}=z_i^{m-1}\sigma_1-z_i^{m-2}\sigma_2+z_i^{m-3}\sigma_3-\dots+(-1)^{m+1}\sigma_m.
\]
\end{proof}

\begin{pro}\label{regular sequence}
Let $G$ be a simple graph on $N$ vertices, $S=\Po/I(G)$, and $d=\dim S$. Let $A_j$ denote the set of 
independent sets of size $j$ in $G$. Then the sequence 
\[
\Big\{\sum_{\{x_i\}\in A_1} x_i,\sum_{\{{x_{i_1},x_{i_2}\}}\in A_2} x_{i_1}x_{i_2},\ldots,\sum_{\{{x_{i_1},\ldots,x_{i_d}\}}\in A_d} x_{i_1}\cdots x_{i_d} \Big\}
\]
is a homogeneous system of parameters for $S$. In particular, if $G$ is CM then this sequence is a regular sequence for $S$.
\end{pro}
\begin{proof}
Let $F_k=\sum_{\{{x_{i_1},\dots,x_{i_k}\}}\in A_k} x_{i_1}\cdots x_{i_k}$, for $1 \leq k\leq d$. Let $\sigma_1,\dots,\sigma_N\in\Po$ be the elementary 
symmetric polynomials. Since the independence number of $G$ is $d$, we have $[\sigma_k]=[0]$ in $S$ for every $k=d+1,\dots,N$. In addition, 
$[\sigma_k]=[F_k]$ for every $k \in \{1,\dots,d\}$. The proposition then follows from the previous lemma:
\[
\sqrt{\langle [F_1],\ldots,[F_d] \rangle}=\sqrt{\langle [\sigma_1],[\sigma_2],\ldots,[\sigma_N] \rangle}=\langle [x_1],\ldots,[x_N] \rangle.
\]
\end{proof}

\section*{Acknowledgements}

We would like to thank all participants of the seminar ``Gr\'aficas y anillos" held at Universidad Aut\'onoma de Zacatecas during the academic year 2014-2015. This work was inspired by this seminar. The third author would like to thank Jos\'e Luis Cisneros for stimulating discussions. The fourth author was partially supported by PIFI-UAZ, UAZ-CA-169.

\vspace{.5cm}
{\tiny \textsc {H. de Alba, Catedr\'atico CONACYT-UAZ.} Email: hdealbaca@conacyt.mx}\\
{\tiny \textsc {W. Carballosa, Catedr\'atico CONACYT-UAZ.} Email: wcarballosato@conacyt.mx}\\
{\tiny \textsc {D. Duarte, Catedr\'atico CONACYT-UAZ.} Email: adduarte@matematicas.reduaz.mx}\\
{\tiny \textsc {L. M. Rivera, Unidad Acad\'emica de Matem\'aticas, UAZ.} Email: luismanuel.rivera@gmail.com}\\
{\tiny \textsc {Universidad Aut\'onoma de Zacatecas (UAZ), Calzada Solidaridad y Paseo La Bufa, Col. Hidr\'aulica, C.P. 98060, Zacatecas, Zacatecas.}}

\end{document}